\documentclass[12pt]{amsproc}

\usepackage{amsmath, amssymb}

\newcommand{\X}{\mathcal{X}}
\newcommand{\F}{\Delta (G)} 
\newcommand{\z}{y}
\newcommand{\y}{z}
\newtheorem{theorem}{\sc Theorem}[section]
\newtheorem{lemma}[theorem]{\sc Lemma}
\newtheorem{proposition}[theorem]{\sc Proposition}
\newtheorem{corollary}[theorem]{\sc Corollary}

\newcommand{\Y}{\mathcal{Y}}
\begin{document}

\title[Restricted centralizers]{Profinite groups with restricted centralizers of commutators}

\author[E. Detomi, M. Morigi, P. Shumyatsky]
{\textbf{Eloisa Detomi}
\\
Dipartimento di Matematica, Universit\`a di Padova\\
 Via Trieste 63, 35121 Padova, Italy\\
(detomi@math.unipd.it)\\
\textbf{Marta Morigi}\\
Dipartimento di Matematica, Universit\`a di Bologna\\
Piazza di Porta San Donato 5, 40126 Bologna, Italy
(marta.morigi@unibo.it)\\
\textbf{Pavel Shumyatsky}\\
Department of Mathematics, University of Brasilia\\
Brasilia-DF, 70910-900 Brazil
(pavel@unb.br)\\
}

 \begin{abstract} 
A group $G$  has restricted centralizers if for each $g$ in $G$ the centralizer $C_G(g)$ either is finite or has finite index in $G.$
 A theorem of Shalev states that a profinite group with restricted centralizers is abelian-by-finite. In the present article we handle profinite groups with restricted centralizers of word-values. We show that if $w$ is a multilinear commutator word and $G$ a profinite group with restricted centralizers of $w$-values, then the verbal subgroup $w(G)$ is abelian-by-finite.
\medskip

{\noindent {\it Keywords:} group words, profinite groups, centralizers, FC-groups
\medskip

}

 \end{abstract}

\maketitle
{\centering
{\it
 Dedicated to Aner Shalev on the occasion  of his $60$th birthday.\par}}

\section{Introduction} 
A group $G$ is said to have restricted 
centralizers if for each $g$ in $G$ the centralizer $C_G(g)$  either is finite or has finite index in $G$. 
This notion was introduced by Shalev in \cite{shalev} where he showed that a profinite group with restricted centralizers is finite-by-abelian-by-finite. Note that a finite-by-abelian profinite group is necessarily abelian-by-finite so Shalev's theorem essentially states that a profinite group with   restricted centralizers is 
abelian-by-finite.

In the present article we handle profinite groups with  restricted centralizers of word-values.
 Given a word $w$ and a  group $G$, we denote by $G_w$ the set of all values of $w$ in $G$ and by $w(G)$ the subgroup generated by $G_w$. In the case where $G$ is a profinite group $w(G)$ denotes the subgroup topologically generated by $G_w$. 
 
 Recall that 
 multilinear commutator words are words which are obtained by nesting commutators, but using always different variables.
 Such words are also known under the name of outer commutator words and are precisely the words that can be written in the form of multilinear Lie monomials.
 
 The main purpose of this paper is to prove the following theorem. 

\begin{theorem}\label{main}
Let $w$ be a multilinear commutator word and $G$ a profinite group 
 in which all centralizers of $w$-values 
 are either finite or open. Then $w(G)$ is abelian-by-finite. 
\end{theorem} 
 
From the above theorem we can deduce the following results.

\begin{corollary}\label{openT} 
Under the hypothesis of Theorem \ref{main}, the group $G$ has an open subgroup $T$ such that $w(T)$ is abelian. 
 In particular $G$ is soluble-by-finite. 
\end{corollary} 

\begin{corollary}\label{profinite-finite} 
 Let $w$ be a multilinear 
commutator word and $G$  a profinite group in which every nontrivial $w$-value has finite centralizer.  Then either $w(G)=1$ or $G$ is finite. 
\end{corollary} 

The proof of Theorem \ref{main} is fairly complicated. We will now briefly describe some of the tools employed in the proof.  

 Recall that a group $G$ is an FC-group if the centralizer $C_G(g)$ has finite index in $G$ for each $g\in G$. Equivalently, $G$ is 
 an FC-group if each conjugacy class $g^G$ is finite. A group $G$ is a BFC-group if all  conjugacy classes in $G$ are finite and have bounded size. 
 A famous theorem of  B. H. Neumann says that the commutator subgroup of a BFC-group is finite  \cite{bhn}. 
 Shalev used this to show that a profinite FC-group has finite commutator subgroup  \cite{shalev}. 
 
 In Section \ref{sec:FC} we generalize Shalev's result by showing that if $w$ is a multilinear commutator word and $G$ is a profinite group in which all $w$-values are FC-elements, then $w(G)$ has finite commutator subgroup. 
 In fact, we establish a much stronger result involving the marginal subgroup introduced by P. Hall (see Section \ref{sec:FC} for details). 
 The results of Section  \ref{sec:FC} enable us to reduce Theorem \ref{main} to the case where all $w$-values have finite order. 
 
 A famous result by Zelmanov says that periodic profinite groups are locally finite \cite{z:periodic}. 
 Recall that a group is said to locally have some property if all its finitely generated subgroups have that property. 
 There is a conjecture stating  that for any word $w$ and any profinite group $G$ in which all $w$-values have finite order, 
 the verbal subgroup $w(G)$ is locally finite. The conjecture is known to be correct in a number of particular cases (see \cite{Sh01,KS, DMS-2015}). 
  In Section \ref{sec:pro-p} we obtain another result in this direction. Namely,   let $p$ be a prime, $w$ a multilinear commutator word and $G$  a profinite group in which all $w$-values have finite $p$-power order. We prove that the abstract subgroup generated by all $w$-values is locally finite.   
 
 The proof of the above result relies on the techniques created by Zelmanov in his solution of the  Restricted Burnside Problem  \cite{Z}. 
 While the result falls short of proving that $w(G)$ is  locally finite, it will be shown to be sufficient for the purposes  of the present paper. 
 Indeed, in Section \ref{sec:locfin} we prove that if  a profinite group $G$ satisfies the hypotheses of Theorem \ref{main} 
and  has all $w$-values of finite order, then $w(G)$ is locally finite. This is achieved by combining results of previous sections with the ones obtained in \cite{KS} and \cite{DMS-2015}. 
 
 In Section \ref{sec:final} we finalize the proof of Theorem \ref{main}. At this stage without loss of generality we can assume that $w(G)$ is locally finite and at least one $w$-value has finite centralizer. 
 With these assumptions, a whole range of tools (in particular, those using the classification of finite simple groups) become available. 
  We appeal to Wilson's theorem  on the structure of compact torsion groups which implies that in our situation $w(G)$    
   has a finite series of closed characteristic subgroups in which each factor 
 either is a pro-$p$ group for some prime $p$ or is isomorphic (as a topological
group) to a Cartesian product of 
 finite simple groups. 
 
 Recall that the famous Ore's conjecture, stating that every element of a nonabelian finite simple group is a commutator, was proved in \cite{lost}. It follows that for each  multilinear commutator word $w$ every element of a nonabelian  finite simple group is a $w$-value. If a group  $K$ is isomorphic to a Cartesian product of 
  nonabelian finite simple groups and has restricted centralizers of $w$-values, then actually all centralizers of elements in $K$ are either finite or of finite index and so, by Shalev's theorem \cite{shalev}, $K$ is finite. 
  We use this observation to conclude that under our assumptions the verbal subgroup is (locally soluble)-by-finite.   
  Finally, an application of the results on FC-groups obtained in Section \ref{sec:FC} completes the proof of Theorem \ref{main}. 
  
  The next section contains a collection of mostly well-known auxiliary lemmas which are used throughout the paper. In Section \ref{sec:comb} we describe combinatorial techniques developed in  \cite{GuMa, DMS-2015,DMS-revised}   for handling   multilinear commutator words. We also prove some new lemmas which are necessary for the purposes of the present article. 
  Throughout the paper, unless explicitly stated otherwise, subgroups of profinite groups are assumed closed.

\section{Auxiliary lemmas}

Multilinear commutator words are words which are obtained by nesting commutators, but using always different variables. More formally, the word $w(x) = x$ in one variable is a multilinear commutator; if $u$ and $v$ are  multilinear commutators involving different variables then the word $w=[u,v]$ is a multilinear commutator, and all multilinear commutators are obtained in this way.
 
An important family of  multilinear commutator words is formed by so-called derived words $\delta_k$,  
  on $2^k$ variables,   defined recursively by
$$\delta_0=x_1,\qquad \delta_k=[\delta_{k-1}(x_1,\ldots,x_{2^{k-1}}),\delta_{k-1}(x_{2^{k-1}+1},\ldots,x_{2^k})].$$
 Of course  $\delta_k(G)=G^{(k)}$ is the $k$-th  term of the derived series of $G$. 

We recall the following well-known result (see for example \cite[Lemma 4.1]{Sh00}). 
\begin{lemma}\label{lem:delta_k} Let $G$ be a group and let $w$ be a multilinear commutator word on $n$ variables. Then each $\delta_n$-value is a $w$-value.
\end{lemma}

The following is Lemma 4.2 in \cite{Sh00}

\begin{lemma}\label{lem:4.2}
Let $w$ be a multilinear commutator word and $G$ a soluble group in which all $w$-values have
finite order. Then the verbal subgroup $w(G)$ is locally finite.
\end{lemma}

If $x$ is an element of a group $G$, we write $x^G$ for the conjugacy class of $x$ in $G$.  More generally,  if $S$ is a subset 
of $G$,
we write $S^G$ for the set of conjugates of elements of $S$. On the other hand, if $K$ is a subgroup of $G$, 
then $K^G$ denotes the normal closure of $K$ in $G$, that is, the subgroup generated by all conjugates of $K$ in $G$, 
with the usual convention
that if $G$ is a topological group then $K^G$ is a closed subgroup.

Recall that if $G$ is a group, $a\in G$ and $H$ is a subgroup of $G$, then $[H,a]$ denotes the subgroup of $G$ generated by all commutators of the form $[h,a]$, where $h\in H$. It is well-known that $[H, a]$ is normalized by $a$ and $H$.

We will denote by $\F $  the set of FC-elements of $G$, i.e. 
$$\F =\{ x\in G \mid |x^G| < \infty\}.$$ 
Obviously  $\F$ is a normal  subgroup of $G$. Note that if $G$ is a profinite group, $\F$ needs not  be closed. 

\begin{lemma} \label{2.3b} 
Let  $G$ be a group. For every $x\in \F$ 
 the subgroup  $[\F,x]^G$ is finite. 
\end{lemma}
\begin{proof} Let $\Delta=\Delta(G)$. Note that $\Delta'$ is locally finite (see \cite[Section 14.5]{rob}).  The subgroup $[\Delta,x]$ is generated by finitely many commutators $[y,x]$ where $y \in \Delta$. Hence $[\Delta,x]$ is finite. 
Further, each commutator $[y,x]$ is an FC-element and so $C_G([\Delta,x])$ has finite index in $G$. Consequently, $[\Delta,x]^G$ is a product of finitely many conjugates of $[\Delta,x]$. The conjugates of $[\Delta,x]$ normalize each other so  $[\Delta,x]^G$ is finite. 
\end{proof}

\begin{lemma}\label{114} 
Let $G$ be a locally nilpotent group containing an element with finite centralizer. Suppose that $G$ is residually finite. Then $G$ is finite.
\end{lemma}
\begin{proof} Choose $x\in G$ such that $C_G(x)$ is finite. Let $N$ be a normal subgroup of finite index such that $N\cap C_G(x)=1$. Assume that $N\neq1$ and let $1\neq y\in N$. The subgroup $\langle x,y\rangle$ is nilpotent and so the center of $\langle x,y\rangle$ has nontrivial intersection with $N$. This is a contradiction since $N\cap C_G(x)=1$.
\end{proof}
 
Lemma 1.6.1 in \cite{Kh} states that if $G$ is a finite group, $N$ is a normal subgroup of $G$ and $x$ an element of $G$, 
then $|C_{G/N}(xN)| \le |C_G(x)|$.   We will need a version of this lemma for locally finite groups.
\begin{lemma}\label{KK} 
 Let $G$ be a locally finite group and $x$ an element of $G$ such that $C_G(x)$ is finite of order $m$. 
 If $N$ is a normal subgroup of $G$, then 
 $|C_{G/N}(xN)| \le m$.  
\end{lemma}
\begin{proof}
Arguing by contradiction, assume that $C_{G/N}(xN)$ contains $m+1$ pairwise distinct elements $b_1N, \dots , b_{m+1}N$. 
Let   $K=\langle x, b_1, \dots , b_{m+1} \rangle$ and $N_0=N \cap K$. 
Note that $K$ is a finite group 
 and $C_{K/N_0} (x N_0)$ contains the  $m+1$ distinct elements $b_1N_0, \dots , b_{m+1}N_0$. This contradicts Lemma 1.6.1 in \cite{Kh}. 
\end{proof}

\begin{lemma}\label{sol1} Let $d,r,s$ be positive  integers.
 Let $G$ be a soluble group of derived length $d$ generated by a set $X$ such that every element in $X$ has finite order dividing $r$ and has at most
 $s$ conjugates in $G$. Then $G$ has finite exponent bounded by a function of $d,r,s$.
\end{lemma}
\begin{proof}
 The proof is by induction on the derived length of $G$. If $G$ is abelian then $G$ has exponent dividing  $r$. 
 Note that $G'$ 
  is generated by all conjugates of the set $\{[y,z]|y,z\in X\}$. As $y,z\in X$ have at most $s$ conjugates in $G$ it follows that $[y,z]$
 has at most $s^2$
conjugates. Note that the center of $\langle y,z\rangle$ coincides with 
 $C_{\langle y,z\rangle}(y)\cap C_{\langle y,z\rangle}(z)$
so it has index at most $s^2$, thus the order of  the derived subgroup of 
$\langle y,z\rangle$ is bounded by  a function of $s$ by Schur's theorem \cite[10.1.4]{rob}. By induction, the exponent of $G'$ is finite and bounded by a function of $d,r,s$.
As $G/G'$ has exponent at most $r$, the result follows.
\end{proof}

Throughout the paper, we will use without explicit references the following result.

\begin{lemma}\label{abelian-by-finite} Let $G$ be a finite-by-abelian profinite group. Then $G$ is central-by-finite.
 \end{lemma}
\begin{proof}
 Let $T$ be a finite normal subgroup of $G$ such that $G/T$ is abelian and let $N$ be an open  normal subgroup 
 of $G$ such that $N\cap T=1$. Then  $N\cap G'=1$ and so $N$ is central in $G$.
\end{proof}


\section{Combinatorics of commutators}\label{sec:comb}
We will need some machinery concerning combinatorics of commutators, so we now recall some notation from the paper \cite{DMS-revised}.
 
 Throughout this section, $w=w(x_1,\dots,x_n)$ will be a fixed multilinear commutator word.
 If $A_1,\dots,A_n$ are subsets of a group $G$, we write
$$\X_w(A_1,\dots,A_n)$$ to denote the set of all $w$-values $w(a_1,\dots,a_n)$ with $a_i\in A_i$. Moreover, we write 
$w(A_1, \dots , A_n)$ for the subgroup $\langle\X_w(A_1,\dots,A_n)\rangle$. Note that if every $A_i$ is a 
normal subgroup of $G$, then $w(A_1, \dots , A_n)$ is normal in $G$.

Let $I$ be a subset of $\{1,\dots,n\}$. 
  Suppose that we have a family $A_{i_1}, \dots , A_{i_s}$ of subsets of $G$ with indices running  over $I$ and another family 
  $B_{l_1}, \dots , B_{l_t}$ of subsets with indices  running  over $\{1, \dots ,n \} \setminus I.$ 
 We write 
 $$w_I(A_i ; B_l)$$ 
 for $w(X_1, \dots , X_n)$, where $X_k=A_k$ if $k \in I$, and $X_k=B_k$  otherwise. 
 On the other hand, whenever $a_i\in A_i$ for $i\in I$ and $b_l\in B_l$ for $l\in \{1,\dots,n\}\setminus I$, the symbol 
 $w_I(a_i;b_l)$ stands for the element $w(x_1, \dots , x_n)$, where $x_k=a_k$ if $k \in I$, and $x_k=b_k$ otherwise.

 The following lemmas are Lemma 2.4, Lemma 2.5 and Lemma 4.1 in \cite{DMS-revised}.
 \begin{lemma}\label{2.1-conjugates} 
Let $w=w(x_1, \dots, x_n)$ be a multilinear commutator word. 
Assume that  $H$ is a normal subgroup of a group $G$. 
Let $ g_1, \dots , g_n \in G$, $h \in H$  and fix $s \in \{1, \dots, n\}$. 
Then there exist 
 $y_j \in g_j^H$,  for  $j=1,\dots, n$, such that 
 \begin{eqnarray*}
 w_{\{s\}}(g_sh; g_l)=w(y_1, \dots,y_n) w_{\{s\}}(h; g_l). 
\end{eqnarray*}
\end{lemma}

 \begin{lemma}\label{2.2-bis}
Let $G$ be a group and  let $w=w(x_1, \dots, x_n)$ be a multilinear commutator word. 
 Assume that $M, A_1,\dots,A_n$ are  normal subgroups of $G$ such that for some elements  
$a_i\in A_i$,  the  equality  
$$w(a_1(A_1\cap M), \dots , a_n(A_n\cap M))=1$$ holds. 
Then for any subset $I$ of $\{1,\dots,n\}$ we have 
$$w_I(A_i\cap M ; a_l (A_l\cap M))=1.$$
\end{lemma}

 \begin{lemma}\label{M2} 
 Let $G$ be a group and  let $w=w(x_1, \dots, x_n)$ be a multilinear commutator word. 
Let $A_1,\dots,A_n$ and $M$ be normal subgroups of   $G$. 
Let $I$ be a subset of $\{1, \dots ,n \}$. Assume that 
\[ w_J (A_i; A_l\cap M)=1\]
for every proper subset $J$ of $I$.  
Suppose we are given elements  $g_i \in A_i$ for $i \in I$ and  elements $h_k \in A_k\cap M$ for $k \in \{1, \dots, n\}$. 
 Then we have 
\[w_I(g_ih_i; h_l)=w_I(g_i;h_l).\] 
\end{lemma}

\begin{lemma}\label{comb1}
 Let $w=w(x_1, \dots, x_n)$ be a multilinear commutator word. 
Assume that  $T$ is a normal subgroup of a group $G$ and 
 $a_1, \dots , a_n$ are elements of $G$ such that every element in $\X_w(a_1T,\dots,a_nT)$ has at most $m$ conjugates in $G$.
Then every element in $T_w$ has at most $m^{2^n}$ conjugates in $G$.
\end{lemma}
\begin{proof}
 We will first prove the following statement:
 \smallskip
 
 {\noindent ($*$)
 Assume that for some $ g_1, \dots , g_n \in G$ 
    every element in the set $\X_w(g_1T,\dots,g_nT)$ has at most $t$ conjugates in $G$,
  and  let $s \in \{1, \dots, n\}$. Then every element of the form
 $w_{\{s\}}(h_s; g_lh_l)$, where $h_1,\dots,h_n\in T$, has at most $t^2$ conjugates.}
\vskip8pt
 
 Choose an element $z=w_{\{s\}}(h_s; g_lh_l)$ as above. By Lemma \ref{2.1-conjugates} 
  \begin{eqnarray*}
 w_{\{s\}}(g_sh_s; g_lh_l)=w(y_1, \dots,y_n) w_{\{s\}}(h_s; g_lh_l),
\end{eqnarray*}
where $y_j \in (g_jh_j)^T\subseteq g_jT$,  for  $j=1,\dots, n$.

As both $w_{\{s\}}(g_sh_s; g_lh_l)$ and $w(y_1, \dots,y_n)$ lie in 
$\X_w(g_1T,\dots,g_nT)$, they have at most $t$ 
conjugates in $G$.
Thus 
$$z=w(y_1, \dots,y_n)^{-1}w_{\{s\}}(g_sh_s; g_lh_l)$$
 has at most $t^2$ conjugates in $G$. This proves ($*$).

We will now prove that  every element in 
 $$\X_w(T,\dots,T,a_{i}T,\dots,a_nT)$$ 
  has at most $m^{2^i}$ conjugates, by induction on $i$. The lemma will follow by taking $i=n$.
 
 If $i=1$ the statement is true by the hypotheses. So assume that $i\ge 2$ and every element in 
 $\X_w(T,\dots,T,a_{i-1}T,\dots,a_nT)$ has at most $m^{2^{i-1}}$ conjugates. By applying ($*$) with
 $g_1=\dots=g_{i-1}=1$, $t=2^{i-1}$ and $s=i$ we get the result.
\end{proof}
\begin{lemma}\label{comb2}
 Let $w=w(x_1, \dots, x_n)$ be a multilinear commutator word. 
Assume that  $H$ is a normal subgroup of a group $G$. Then there exixst a positive integer $t_n$ depending only on $n$ 
such that for every 
$ g_1, \dots , g_n \in G$, $ h_1, \dots , h_n \in G$ the $w$-value $w(g_1h_1,\dots,g_nh_n)$ can be written in the form:
$w(g_1h_1,\dots,g_nh_n)=ah,$
where $a$ is a product of at most $t_n$ conjugates of elements in  $\{g_1^{\pm 1}, \dots , g_n^{\pm 1}\}$ and $h\in H_w$. 
\end{lemma}
\begin{proof}
 The proof is by induction on the number $n$ of variables appearing in $w$. If $n=1$ then $w=x$ and the result is true. 
 
 If $n>1$,
 then $w$ is of the form $w=[u,v]$, where $u=u(x_1,\dots,x_r)$, $v=v(x_{r+1},\dots,x_n)$ are multilinear commutator words.
 By induction, $u(g_1h_1,\dots,g_rh_r)=a_1h_1$, $v(g_{r+1}h_{r+1,}\dots,g_nh_n)=a_2h_2$, 
 where $a_1$ (resp. $a_2$) is a product of at most $t_r$ (resp. $t_{n-r}$) conjugates of elements in  
 $S=\{g_1^{\pm 1}, \dots , g_n^{\pm 1}\}$, $h_1\in H_u$ and $h_2\in H_v$. By the standard commutator formulas we have that:
 $$w(g_1h_1,\dots,g_nh_n)=[a_1h_1,a_2h_2]=[a_1,a_2h_2]^{h_1}[h_1,a_2h_2]=$$ 
 $$
 [([a_1,h_2][a_1,a_2]^{h_2})^{h_1}[h_1,h_2][h_1,a_2]^{h_2}=$$
 $$[a_1,h_2]^{h_1}[a_1,a_2]^{h_2h_1}[h_1,a_2]^{h_2}[h_1,h_2]^{[h_1,a_2]^{h_2}},$$
where $[a_1,h_2]=a_1^{-1}a_1^{h_2}$, $[a_1,a_2]=a_1^{-1}a_1^{a_2}$ are products of at most $2t_r$ conjugates of elements in $S$,
$[h_1,a_2]=(a_2^{-1})^{h_1}a_2$ is a product of at most $2t_{n-r}$ conjugates of elements in $S$ and $[h_1,h_2]^{[h_1,a_2]^{h_2}}\in H_w$.
So the result follows taking $t_n$ to be the maximum of the set $\{4t_r+2t_{n-r}|r=1,\dots,n-1\}$. 
\end{proof}


\section{ Profinite groups in which  $w$-values are  FC-elements. }\label{sec:FC}

The famous theorem of  B. H. Neumann says that the commutator subgroup of a BFC-group is finite  \cite{bhn}.
 This was recently extended in \cite{DMS-BFC} as follows. 
  Let $w$ be a multilinear commutator word and $G$ a group in which 
  $|x^G|\le m$ for every $w$-value $x$. Then the derived subgroup of $w(G)$ is finite of order bounded by a 
function of $m$ and $w$. The case where $w=[x,y]$ was handled in \cite{dieshu}. 

In the present article we require a profinite (non-quantitative) version of the above result. We show that if $G$ is a profinite group in which all $w$-values are FC-elements, then  
the derived subgroup of  $w(G)$ is finite. In fact we establish a stronger result, which uses the concept of marginal subgroup.

Let G be a group and $w=w(x_1,\dots,x_n)$   a word. The marginal subgroup $w^*(G)$ of
$G$ corresponding to the word $w$ is defined as the set of all $x \in G$  such that
 $$w(g_1,\dots, x g_i,\dots,g_n)= w(g_1,\dots,  g_i x ,\dots,g_n)=w(g_1,\dots,g_i,\dots,g_n)$$
 for all  $g_1,\dots,g_n \in G$ and $1 \le i \le n$. 
 It is well known that $w^*(G)$ is a characteristic subgroup of $G$ and that  $[w^*(G), w(G)]=1$. 
  
  Note that marginal subgroups in profinite groups are closed. 

Let  $S$ be a subset of a group $G$.
Define the  $w^*$-residual of  $S$ of $G$ to be the intersection  of all normal subgroups $N$ such that $SN/N$ is contained in the marginal subgroup  $w^*(G/N)$.

 For  multilinear commutator words the  $w^*$-residual of  a normal subgroup has the following characterization.   

\begin{lemma}\label{ts}
Let $w$ be a multilinear 
commutator word, $G$  a group and $N$ a normal subgroup of $G$. 
 Then the   $w^*$-residual of $N$ in $G$ is the subgroup generated by the elements $w(g_1, \dots, g_n)$ where at least one of $g_1, \dots, g_n$ belongs to $N$. \end{lemma} 
This follows from  \cite[Theorem 2.3]{TS}. 
For the reader's convenience, we will give here a proof in the spirit of Section \ref{sec:comb}. 

\begin{proof} 
Let $N_i=\langle w(g_1, \dots, g_n)|g_1,\dots,g_n\in G{\textrm{ and }}g_i\in N\rangle$ 
and let $R=N_1N_2\dots N_n$. 
 Clearly, if $M$ is a normal subgroup of $G$ such that $N/M$ is contained in $w^*(G/M)$ then $N_i\le M$ for every
$i=1,\dots, n$. 
 Therefore $R$ is contained in the $w^*$-residual of $N$ 

 On the other hand, it  follows from Lemma \ref{M2} 
 that if  $N_i=1$, then $$w(g_1,\dots,g_ih,\dots,g_n)=w(g_1,\dots,g_i,\dots,g_n)$$ for every $g_1,\dots,g_n\in G$ and 
every $h$ in $N$. 
Thus we have
$$w(g_1,\dots,g_ih,\dots,g_n)R=w(g_1,\dots,g_i,\dots,g_n)R$$
for every $i=1,\dots, n$, for every $g_1,\dots,g_n\in G$ and 
every $h\in N$. So $N/R$ is contained  in $w^*(G/R)$.
This implies the result.
\end{proof}

It follows form Lemma \ref{ts} that if $w$ is a multilinear commutator word and $N$ is a normal subgroup of 
a group $G$ which does not contain nontrivial $w$-values, then $N$ is contained $w^*(G)$ and, in particular, it centralizes $w(G)$. 
 Indeed in this case, by Lemma \ref{ts},  the $w^*$-residual of  $N$ in $G$ is trivial. 
 
   A word $w$ is  concise if whenever $G$ is a group
 such that the set $G_w$ is finite, it follows that also $w(G)$ is finite. Conciseness of multilinear commutators was proved by
 J.\,C.\,R. Wilson in \cite{jwilson}  (see also \cite{GuMa}).
 
\begin{lemma}\label{concise}
 Let $w$ be multilinear 
commutator word, $G$  a profinite group and  $N$  an open normal subgroup of $G$. Then the $w^*$-residual of $N$ is open in $w(G)$. 
\end{lemma}
\begin{proof}
 Let $K$ be the $w^*$-residual of $N$. 
 As $N/K$ is   contained in $w^*(G/K)$ and it has finite index in $G/K$, we deduce that the set of $w$-values of $G/K$ is 
 finite.  It follows from the above result of   Wilson   that 
 $w(G/K)$ is finite, as desired.
\end{proof}

As above,  $\F $ denotes the set of FC-elements of $G$. 
  In what follows we will denote by $H$ the  topological closure of $\F$ in a profinite group $G$.  

 The goal of this section is to  prove   the following theorem. 

\begin{theorem}\label{genN}
 Let $w$ be a multilinear 
commutator word, $G$  a profinite group and $T$ a normal subgroup of $G$ such that 
 every $w$-value of $G$  contained in $T$ is an FC-element. 
 Then the $w^*$-residual of $T$ has finite commutator subgroup.  
\end{theorem}

It is straightforward  that the $w^*$-residual of $G$ is precisely $w(G)$.  
 Thus Theorem \ref{genN} has the following consequence.    
 
\begin{corollary}\label{profinite-FC} 
 Let $w$ be a multilinear 
commutator word and $G$  a profinite group in which every $w$-value is an FC-element.  Then $w(G)$ has finite commutator subgroup.   
\end{corollary} 

 The key result of the remaining part of this section is the next proposition, 
from which Theorem \ref{genN} will be deduced. 

\begin{proposition}\label{X}
 Let $w=w(x_1,\dots,x_n)$ be a multilinear 
commutator word, $G$ a profinite group and $H$ the topological closure of $\F$ in $G$. 
  Assume that  $A_1,\dots, A_n$ are normal subgroups of $G$ with the property that 
 $$ \X_w(A_1,\dots,A_n) \subseteq \F.$$
Then $[H , w(A_1,\dots,A_n) ]$ is finite.
\end{proposition}

The following lemma can be seen as a development related to Lem\-ma 2.4 in \cite{dieshu} and Lemma 4.5 in \cite{wie}.

\begin{lemma} \label{basic-light} 
Assume the hypotheses  of Proposition \ref{X}, 
 with 
  $A_1,\dots,$  $A_n$ being normal subgroups of $G$ with the property that 
 $ \X_w(A_1,\dots,A_n)$ $\subseteq \F.$
Let  $M$ be an open normal subgroup of $G$ and $a_i\in A_i$ for $i=1,\dots,n$. Then there exist elements $\tilde a_i\in a_i (A_i\cap M)$  
 and an open normal subgroup $\tilde M$ of $M$, such that the order of 
 $$[H,w(\tilde a_1 (A_1\cap \tilde M),\dots,\tilde a_n (A_n\cap \tilde M))]^G$$ 
is finite. 
\end{lemma}
\begin{proof}
Throughout the proof, whenever $K$ is a subgroup of $G$ we  write $K_i$ for $ A_i\cap K$.

For each natural number $j$ consider the set $\Delta_j$ of elements $g \in G$ such that $|G:C_G(g)| \le j$.  
 Note that the sets $\Delta_j$ are closed (see for instance   \cite[Lemma 5]{LP}). 
Consider the sets  
$$C_j=\{(\z_1,\dots,\z_n) \mid  \z_i\in a_iM_i {\textrm{ and }} w(\z_1,\dots,\z_n) \in \Delta_j\}.$$
Each set $C_j$ is closed, being  the inverse image in $a_1  M_1 \times \cdots \times a_n M_n$ of the closed set  $\Delta_j$ under the
continuous map $(g_1, \dots , g_n) \mapsto w(g_1, \dots , g_n)$. 
Moreover the union of the sets $C_j$ is the whole $a_1 M_1 \times \cdots \times a_n  M_n$. 
By the Baire category theorem (cf. \cite[p.\ 200]{Ke}) at least one of the sets $C_j$ has nonempty interior. 
Hence, there exist a natural number $m$, some elements
$\y_i\in a_i M_i$ and a normal open subgroup
$Z$ of $G$ such that 
$$w(\y_1 Z_1,\dots,\y_nZ_n)\subseteq \Delta_{m}.$$ By replacing $Z$ with $Z\cap M$, if necessary,
we can assume that $Z\le M$.

 Choose in 
 $\X_w(\y_1Z_1,\dots,\y_n Z_n)$ an element $a=w(\tilde a_1,\dots,\tilde a_n)$ such that the number of 
conjugates of $a$ in $H$ 
is maximal among the elements of 
 $\X_w(\y_1 Z_1,\dots,\y_n Z_n)$, 
that is,  $|a^H|\ge |g^H|$ for any
 $g\in \X_w(\y_1Z_1,\dots,$ $\y_n Z_n)$. 

  Since $\F$ is dense in $H$, 
 we can choose a right transversal   $b_1,\dots, b_r$ of $C_H(a)$ in $H$ consisting of  FC-elements. 
Thus  $a^H = \{a^{b_i} | i = 1, \dots, r\}$, where 
 $a^{b_i}\ne a^{b_j}$ if $i\ne j$.
  Let $ \tilde M$ be the intersection of $Z$ and all $G$-conjugates of $ C_G ( b_1 ,\dots, b_r )$: 
 $$\tilde M  =\left(\bigcap_{g\in G}  C_G (  b_1 ,\dots, b_r )^g \right) \cap Z$$ 
  and note that $\tilde M$ is open in $G$. 

Consider the element  $w(\tilde a_1v_1,\dots,\tilde a_nv_n)$ where   $v_i \in \tilde M_i$ for $i=1,\dots,n$.  
 As $w(\tilde a_1v_1,\dots,\tilde a_nv_n) \tilde M_i=a \tilde M_i$ in the quotient group $G/\tilde M_i$, we 
  have 
$$w(\tilde a_1v_1,\dots,\tilde a_nv_n)=va,$$
 for some $v\in \tilde M\le  C_G ( b_1 ,\dots, b_r )$. 
It follows that $(va)^{b_i} = va^{b_i}$ for each $i =1,\dots, r$. Therefore the elements $va^{b_i}$ form the conjugacy class
 $(va)^H$ 
because they are all different and their number is the allowed maximum. So, for an arbitrary element $h\in H$ there exists $b\in\{b_1 ,\dots, b_r\}$ such that
$(va)^h= va^b$ and hence $v^h a^h = va^b$. Therefore $[h, v] = v^{-h}v=a^h a^{-b}$ and so $[h, v]^a =a^{-1} a^h a^{-b} a = [a,h][b,a] \in [H,a].$
 Thus $[H,v]^a \le [H,a]$ and $$[H, va]=[H,a] [H,v]^a \le [H, a].$$ 
Therefore $[H,w(\tilde a_1 \tilde M ,\dots,\tilde a_n \tilde M)]\le [H,a]$.
Lemma \ref{2.3b} states that the abstract group $[ \Delta(G),a]^G$  
 has finite order and thus the same holds for $[H,a]^G$.
 The result follows.
\end{proof}

For the reader's convenience, the most technical part of the proof of Proposition \ref{X}  is isolated in the following proposition.

\begin{proposition}\label{inductive-step} 
Assume the hypotheses  of Proposition \ref{X}, 
 with 
  $A_1,\dots, A_n$ being normal subgroups of $G$ such that
 $ \X_w(A_1,\dots,A_n) \subseteq \F.$ 
Let $ I $ be a nonempty  subset of $\{1,\dots,n\}$      and assume that there exist a normal subgroup $U$ 
of $G$ of  finite order  and  
an open normal subgroup $M$ of $G$  such that
\[[H, w_J (A_i; A_l\cap M)]\le U \quad \textrm{for every}\ J \subsetneq I.\]
Then there exist a finite normal subgroup $U_I$ of $G$
containing $U$ 
 and  an open normal subgroup $M_I$ of $G$ 
  contained in $M$ such that 
\[ [H,w_I (A_i; A_l\cap M_I)]\le U_I.\]
\end{proposition}
\begin{proof}
 For each $i=1,\dots,n$ consider a  right transversal $C_i$  of $A_i\cap M$ in $A_i$, and 
  let $\Omega$ be the set of $n$-tuples  $\underline{c}=(c_1, \dots , c_n)$ where $c_r \in C_r$ if $r\in I$ and $c_r=1$ otherwise. 
   Note that the set $\Omega$ is finite, since $C_r$  is finite for every $r$. 
 For any  $n$-tuple  $\underline{c}=(c_1, \dots , c_n) \in \Omega$,  by  Lemma \ref{basic-light}, 
 there exist elements $d_i\in  c_i(A_i\cap M)$ and an open normal subgroup $M_{\underline c}$ of  $G$ such that the order of 
$$[H,w(d_1  (A_1\cap M_{\underline c}),\dots,d_n (A_n\cap M_{\underline c}))]^G$$
 is finite.
Let 
\begin{eqnarray*}
M_I&=&M \cap \bigg( \bigcap_{\underline{c}\in \Omega}M_{\underline c}\bigg),\\
U_I&=& U \, \prod_{\underline{c}\in \Omega}[H,w(d_1   (A_1\cap M_{\underline c}),\dots,d_n  (A_n\cap M_{\underline c}))]^G.
\end{eqnarray*}

As $\Omega$ is finite,
 it follows that 
 $M_I$ is open in $G$ and $U_I$ has finite order.

Let $Z/U_I$ be the center of $HU_I/U_I$ in the quotient group $G/U_I$ and let $\bar G=G/Z$.
 We will use the bar notation to denote images of elements or subgroups in the quotient group $\bar G$. 
 
Let us consider an arbitrary generator $w_I(k_i,h_l)$ of $w_I (A_i; A_l\cap M_I)$, where  $k_i \in A_i$ and $h_l \in A_l\cap M_I$.
Let  $\underline{c}=(c_1, \dots , c_n) \in \Omega$ be the $n$-tuple such that 
$$k_i\in c_{i}(A_i \cap M)$$
 if $i\in I$ and $c_i=1$ otherwise. 
 Let $d_1,\dots,d_n$ be the elements as above, corresponding to the  $n$-tuple  $\underline{c}$. 
 Then, by definition of $U_I$,  
$$[H,w(d_1 (A_1\cap M_I),\dots,d_n (A_n\cap M_I))]\le U_I,$$
 that is 
$$\overline{w(d_1 (A_1\cap  M_I),\dots,d_n (A_n\cap M_I))}=1,$$
 in the quotient group $\bar G=G/Z$. 
 We deduce from  Lemma \ref{2.2-bis}  that 
\begin{equation}\label{step}
 \overline{w_I(d_i (A_i\cap M_I);(A_l\cap M_I))}=1. 
\end{equation}
Moreover, as $c_{i}(A_i\cap M)=d_i(A_i\cap M)$, we have that $k_i=d_iv_i$ for some $v_i\in A_i\cap M$.
It also follows from our assumptions that  $$\overline{w_J (A_i; A_l\cap M)}=1$$
for every proper subset $J$ of $I$. Thus we can apply Lemma \ref{M2} and  obtain that 
$$w_I (\overline k_i;\overline h_l)=w_I (\overline d_i\overline v_i;\overline h_l)=
w_I (\overline d_i;\overline h_l)=1,$$
where in the last equality we have used (\ref{step}).
Since $w_I(k_i,h_l)$ was an arbitrary generator of $ w_I (A_i; A_l\cap M_I)$, 
 it follows that $$\overline{w_I (A_i; A_l\cap M_I)}=1,$$
  that is \[ [H,w_I (A_i; A_l\cap M_I)]\le U_I,\]
as desired.
\end{proof}

\begin{proof}[Proof of Proposition \ref{X}.]
Recall that $w=w(x_1,\dots,x_n)$ is a multilinear 
commutator word, $G$ is a profinite group,  $H$ is the closure of  $\F$ and  $A_1,\dots, A_n$ are normal subgroup of $G$ with the property that 
 $$ \X_w(A_1,\dots,A_n) \subseteq \F.$$
We want to prove that $[H , w(A_1,\dots,A_n) ]$ is finite.

We will prove that for every $s=0,\dots,n$  there exist a finite normal subgroup $U_s$ of $G$ and
an open normal subgroup $M_s$ of $G$  such that whenever $I$ is a subset of $\{1,\dots,n\}$ of size at most $s$ we have 
\[ [H,w_I (A_i; A_l\cap M_s)]\le U_s.\]
Once this is done, the proposition  will follow taking $s=n$.

Assume that $s=0$. 
We apply Lemma \ref{basic-light} with $M=G$ and $a_i=1$ for every $i=1,\dots,n$. 
Thus  there exist  $\tilde a_1,\dots, \tilde a_n \in G$ and  an open normal subgroup   $M_0$ of $G$, such that the order of 
$$U_0=[H,w( \tilde a_1 (A_1\cap M_0),\dots , \tilde a_n (A_n\cap M_0)]^G$$
 is finite.

Let $Z/U_0$ be the center of $HU_0/U_0$ in the quotient group $G/U_0$ and let $\bar G=G/Z$.
We have that 
$$\overline{w( \tilde a_1 (A_1\cap M_0),\dots , \tilde a_n(A_n\cap M_0))}=1,$$
 so it follows from Lemma \ref{2.2-bis} that
$$\overline{w( A_1\cap M_0 ,\dots , A_n\cap M_0)}=1,$$
 that is, $[H,w( A_1\cap M_0 ,\dots ,A_n\cap M_0)]\le U_0$. This proves the proposition in the case where $s=0$.

Now assume that $s\ge 1$. Choose
$I\subseteq\{1,\dots,n\}$ with $|I|=s$. By induction, the hypotheses of Proposition \ref{inductive-step} are satisfied with $U=U_{s-1}$
and $M=M_{s-1}$, so there exist a finite normal subgroup $U_I$ of $G$  containing $U_{s-1}$ 
  and   an open normal subgroup $M_I$ of $G$  contained in $M_{s-1}$    such that
\[ [H,w_I (A_i; A_l\cap M_I)]\le U_I.\]
Let
 $$M_s=\bigcap_{|I|=s}M_I, \quad U_s=\prod_{|I|=s}U_I,$$
  where the intersection (resp. the product) ranges over all subsets $I$ of $\{1,\dots,n\}$ of size $s$.

As there is a finite number of choices for $I$, it follows that $U_s$ (resp. $M_s$) has  finite order (resp. finite index 
in $G$).  Note that $M_s\le M_{s-1}$ and $U_{s-1}\le U_s$.
Therefore
\[ [H,w_I (A_i; A_l\cap M_s)]\le U_s\] 
 for every $I\subseteq\{1,\dots,n\}$ with $|I|\le s$. 
This completes the induction  and the proof of the proposition.
\end{proof}

\begin{proof}[Proof of Theorem \ref{genN}.]
 Let $w=w(x_1,\dots,x_n)$ be a multilinear 
commutator word, $G$ a profinite group and $T$ a normal subgroup of $G$.  
 For $i=1, \dots n$, let $X_i$ be the set of $w$-values   $w(g_1,...,g_n)$ such that $g_i$ belongs to $T$. 
 Obviously $X_i  \subseteq T$ and therefore 
  $X_i \subseteq \F$ for every $i$. It follows   from Proposition \ref{X} that $[H, \langle X_i \rangle]$ is finite for every $i$. 
  By Lemma \ref{ts},   the $w^*$-residual of $T$ is the subgroup $N$ generated by the set $X= X_1 \cup \dots  \cup X_n $. 
 Thus $[H, N]= \prod_{i=1}^{n}[H,\langle X_i \rangle]$ is finite. Finally, note that $N\le H$ and so $N'\le [H,N]$ is also finite.
 \end{proof}

\begin{corollary}\label{infinite}
Let $w$ be a multilinear commutator word and let $G$ be a profinite group with restricted centralizers of $w$-values.
 If $G$ has a $w$-value of infinite order, then $w(G)$ is abelian-by-finite. 
\end{corollary} 
\begin{proof}  
Let $x$ be a $w$-value of $G$ of infinite order. As $C_G(x)$ is open, it contains an open normal subgroup $C$ of $G$. 
Let  $K$ be the $w^*$-residual of $C$ in $G$.  
Since all $w$-values contained in $C$ have infinite centralizers, we apply Theorem \ref{genN}  and conclude that $K'$ is finite. 
Being finite-by-abelian,  $K$  is also abelian-by-finite.
  It follows from Lemma \ref{concise} that $K$ has finite index in  $w(G)$ and so $w(G)$ is abelian-by-finite.
\end{proof} 

\section{Pronilpotent groups with restricted centralizers of $w$-values}\label{sec:pro-p}

In the present section we use the techniques   created by  Zelmanov to deduce a theorem about pronilpotent groups with restricted centralizers of $w$-values 
(see Theorem \ref{pro-p}). A combination of this result with Corollary \ref{infinite} yields a proof of Theorem \ref{main} for pronilpotent groups.


For the reader's convenience we collect some definitions and facts on Lie algebras associated with groups (see \cite{S-Lie} or \cite{Z} for further information).
Let $L$ be a Lie algebra over a field. We use the left normed notation; thus if $l_1,\dots,l_n$ are elements of $L$ then $$[l_1,\dots,l_n]=[\dots[[l_1,l_2],l_3],\dots,l_n].$$
An element $y\in L$ is called ad-nilpotent if $ad\, y$ is nilpotent, i.e. there exists a positive integer $n$ such that $[x,{}_ny]=0$ for all $x\in L$. If $n$ is the least integer with the above property then we say that $y$ is ad-nilpotent of index $n$. Let $X$ be any subset of $L$. By a commutator in elements of $X$ we mean any element of $L$  that could be obtained from elements of $X$ by repeated operation of commutation with an arbitrary system of brackets, including the elements of $X$. 
Here the elements of $X$ are viewed as commutators of weight 1. Denote by $F$ the free Lie algebra over the same field 
as $L$ on countably many free generators $x_1,x_2,\dots$. Let $f=f(x_1,\dots,x_n)$ be a nonzero element of $F$. 
The algebra $L$ is said to satisfy the identity $f\equiv 0$ if $f(a_1,\dots,a_n)=0$ for any $a_1,\dots,a_n\in L$. 
In this case we say that $L$ is $PI$. We are now in position to quote a theorem of Zelmanov \cite{Z,ze3} which
has numerous important applications to group theory. A detailed proof of this result recently  appeared in \cite{ze4}.

\begin{theorem} \label{zzz}
Let $L$ be a Lie algebra generated by finitely many elements $a_1, 
 \dots,a_m$ such that all commutators in $a_1, 
 \dots,a_m$ are ad-nilpotent. If $L$ is $PI$, then it is nilpotent.
\end{theorem}

Let $G$ be a group.
Recall that the lower central word $[x_1,\ldots,x_k]$ is usually denoted by $\gamma_{k}$. The corresponding verbal subgroup $\gamma_k(G)$ is the familiar $k$th term of the lower central series of the group $G$.  
 Given a prime $p$, a Lie algebra can be associated with the group $G$ as follows. We denote by 
$$D_i=D_i(G)= \prod_{jp^k \ge i} \left(\gamma_j(G)\right)^{p^k}$$ 
the $i$th dimension subgroup of $G$ in characteristic $p$ (see for example \cite[Chap. 8]{hb}). 
These subgroups form a central series of $G$ known as the Zassenhaus-Jennings-Lazard series. Set $L(G)=\bigoplus D_i/D_{i+1}$. 
Then $L(G)$ can naturally be viewed as a Lie algebra over the field $\mathbb F_p$ with $p$ elements. For an element 
$x\in D_i\setminus D_{i+1}$ we denote by $\tilde x$ the element $xD_{i+1}\in L(G)$.

\begin{lemma}[Lazard, \cite{la}]\label{Laz}
For any $x\in G$ we have $(ad\,\tilde{x})^p=ad\,(\tilde{x^p})$.
\end{lemma}

The next proposition follows from
the proof of the main theorem in the paper of Wilson and Zelmanov \cite{WZ}.

\begin{proposition}\label{prop:WZ}
 Let $G$ be a group satisfying a coset identity. Then $L(G)$ is $PI$.
\end{proposition}

Let $L_p(G)$ be the subalgebra of $L(G)$ generated by $D_1/D_2$. Often, important information about the group $G$ can be deduced from nilpotency of the Lie algebra $L_p(G)$.

 \begin{proposition}\cite[Corollary 2.14]{S-Lie}\label{prop:2.14}
Let $G$ be 
 a group generated by elements $a_1, a_2, \dots , a_m$ such that every $\gamma_k$-value in $a_1, a_2, \dots , a_m$ has finite order, for every $k$.  Assume
that $L_p(G)$ is nilpotent. Then the
series $\{D_i\}$ becomes stationary after finitely many steps.
\end{proposition}
%

Let $P$ be a Sylow subgroup of a finite group $G$. 
An immediate corollary of the Focal Subgroup Theorem \cite[Theorem 7.3.4]{Gor} is that  $G'\cap P$ is generated by commutators.
A weaker version of this fact for multilinear commutator words was proved in  \cite[Theorem A]{AFS}.
\begin{proposition}\label{focal}
  Let $G$  be a finite group and $P$ a Sylow subgroup of $G$. If $w$ is
a multilinear commutator word, then $w(G) \cap P$ is generated by powers of $w$-values.
  \end{proposition}

\begin{proposition}\label{prop:abstract}
Let $p$ be a prime,  $w$  a multilinear commutator word and  $G$  a profinite group in which all $w$-values have finite $p$-power order. 
Let $K$ be the abstract subgroup of $G$ generated by all  $w$-values. Then $K$ is a locally finite $p$-group.
\end{proposition}
\begin{proof} 
It follows from Proposition \ref{focal} that $w(G)$ is a pro-$p$ group. 
Indeed if $Q$ is a Sylow $q$-subgroup of $w(G)$, then 
 the image of $Q$ in any finite continuous image of $G$ is generated 
 by powers of $w$-values, 
which are $p$-elements, hence $Q=1$ unless $q=p$.  

By Lemma \ref{lem:delta_k} there exists an integer $k$  such that each $\delta_k$-value is a $w$-value. 
 It is sufficient to prove that the abstract subgroup $R$ generated by all $\delta_k$-values is locally finite. 
 Indeed, the abstract group $G/R$ is a soluble group such that all $w$-values have finite order. Hence  $w(G/R)$ is locally 
 finite by Lemma \ref{lem:4.2}. 

Let $X$ be the set of $\delta_k$-values of $G$.
 Every finitely generated subgroup of $R$ is contained in a subgroup generated by a finite subset of $X$. 
 So we choose finitely many elements $a_1, \dots, a_s$ in $X$ and consider the subgroup $H$ topologically generated by $a_1, \dots, a_s$. It is sufficient to prove that $H$ is finite. 

Note that $H$ is a pro-$p$ group, since it is a subgroup of $w(G)$. 
For every positive integer $t$, consider the set
$$S_t=\{ (h_1, \dots, h_{2^k}) \mid h_i\in H \ {\textrm{and}}\  \delta_k(h_1, \dots, h_{2^k})^{p^t}=1 \}.$$ 
These sets are closed and their union is the whole Cartesian product of $2^k$ copies of $H$. 
By the Baire category theorem 
 at least one of the sets $S_t$ has nonempty interior. 
Hence,  there exist a natural number $m$, some elements
$y_i\in H$ and a normal open subgroup
$Z$ of $H$ such that 
$$\delta_k(y_1 Z,\dots,y_{2^k}Z)^{p^m}=1.$$ 
In particular $H$ satisfies a coset identity. 

Let $L=L_p(H)$ be the Lie algebra associated with the Zassenhaus-Jennings-Lazard series $\{ D_i \}$. 
 Then $L$ is generated by $\tilde{a}_i=a_i D_2$ for $i=1, \dots , s$. Let $b$ 
any Lie-commutator in $\tilde{a}_1, \dots, \tilde{a}_s$ and let $c$ be the group-commutator in $a_1, \dots, a_s$
having the same system of brackets as $b$. Since $X$ is commutator closed, $c$ is a $\delta_k$-value and so it has finite order.
By Lemma \ref{Laz} this implies that $b$ is ad-nilpotent. 
As $H$  satisfies a coset identity,  it follows from  Proposition \ref{prop:WZ} that $L$ satisfies some nontrivial
polynomial identity. 
By Theorem \ref{zzz} we conclude that $L$ is nilpotent. 
As every $\gamma_k$-value in $a_1, \dots, a_s$  has finite order, 
  Proposition \ref{prop:2.14}  shows that the series $\{D_i\} $ has only finitely many nontrivial terms. 
 Since $H$ is a pro-$p$ group, it follows that the intersection of all $D_i$'s  is trivial. 
 Taking into account that each $D_i$ has finite index in $H$, we deduce that $H$ is finite. This proves that $R$ is locally finite and the proposition follows. 
\end{proof}

\begin{theorem}\label{pro-p}
Let $w$ be a multilinear commutator word and let $G$ be a pronilpotent group with restricted centralizers of $w$-values  in which every $w$-value
has finite order. Then the derived subgroup of $w(G)$ is  finite. 
\end{theorem}
\begin{proof} 
First assume that $G$ is a pro-$p$ group.
Let $K$ be the abstract subgroup of $G$ generated by all $w$-values. By Proposition \ref{prop:abstract} $K$ is a locally finite $p$-group. 
 If a $w$-value of $G$ has finite centralizer, then  $K$ is finite by Lemma \ref{114}. 
Since $K$ is dense in $w(G)$, we conclude that $w(G)$ is finite. Therefore we can assume that 
 every $w$-value in $G$ is an FC-element 
and so the result follows from Corollary \ref{profinite-FC}. 

When  $G$ is pronilpotent, it is the Cartesian product of its Sylow subgroups.
  Let $\mathcal P$ be the set of primes $p$ such that   $w(P) \neq 1$ where $P$ is the Sylow $p$-subgroup  of $G$.  
  If $\mathcal P$ is infinite, then $G$ has a $w$-value of infinite order, against our assumption. Thus $\mathcal P$ is finite. 
  If $P$ is a   Sylow $p$-subgroup of $G$, then the derived subgroup of $w(P)$ is finite by what we proved above. Therefore the derived subgroup of 
    $w(G)=\prod_{p \in \mathcal P} w(P)$  is finite,  as desired. 
\end{proof}


\section{Local finiteness of $w(G)$ }\label{sec:locfin}

The goal of the present section is to show that if the hypotheses of Theorem \ref{main} hold and all $w$-values have finite order, then $w(G)$ is locally finite. 
 There is a long-standing  conjecture stating that each torsion profinite group has finite exponent (cf. Hewitt and Ross \cite{HR}). 
 The conjecture can be easily proved for soluble groups (cf.   \cite[Lemma 4.3.7]{ribes-zal}). In  \cite{DMS-2015} this was extended as follows.

\begin{proposition}\cite[Theorem 3]{DMS-2015}\label{2015} 
 Let $w$ be a multilinear commutator word and $G$ a soluble-by-finite profinite group in which all
$w$-values have finite order. Then $w(G)$ is locally finite and has finite exponent.
\end{proposition}

We remark that the above result does not follow from Lemma \ref{lem:4.2} and its proof is significantly more complicated.

Given a word $w$ and a subgroup $P$ of a profinite group $G$,  we denote by $W(P)$ the  closed subgroup generated by all elements
of $P$ that are conjugate in $G$ to elements of $P_w$: 
$$W(P)= \langle {P_{w}}^G\cap P\rangle.$$
Let $\Y_w$  be the class of
all profinite groups $G$ in which all $w$-values have finite order and the subgroup  $W(P)$
 is periodic for any Sylow subgroup $P$ of $G$.
 
 The following theorem   was implicitly established in \cite{KS}. We will now reproduce the proof. 

\begin{theorem}\label{prop-KS}
Let $w$ be a multilinear commutator word and let $G$ be a profinite group in the class $\Y_w$. Then $w(G)$ is locally finite.
\end{theorem}
\begin{proof}
Recall that finite groups of odd order are soluble by the
Feit-Thompson theorem \cite{FT}. Combining this with \cite[Theorem 1.5]{KS}  (applied with $p = 2$),
we deduce that $G$ has a finite series of closed characteristic subgroups
\begin{equation}\label{5.1} 
G = G_0 \ge  G_1 \ge \cdots \ge G_s = 1 
\end{equation}
in which each factor either is prosoluble or is isomorphic to a Cartesian product
of nonabelian finite simple groups. There cannot be infinitely many nonisomorphic
nonabelian finite simple groups in a factor of the second kind, since this would give a
$w$-value of infinite order. Indeed, by a result of Jones \cite{J}, 
 any infinite family of finite
simple groups generates the variety of all groups; therefore, the orders of $w$-values
cannot be bounded on such an infinite family. Thus, we can assume in addition that
each nonprosoluble factor in (\ref{5.1}) is isomorphic to a Cartesian product of isomorphic
nonabelian finite simple groups. We use induction on $s$.
 If $s = 0$, then $G=1$ and the result follows. 
Let $s \ge 1$. By induction, $w(G_1)$ is locally
finite. Passing to the quotient $G/w(G_1)$, we can assume that $G_1$ is soluble. If $G/G_1$ is
isomorphic to a Cartesian product of isomorphic nonabelian finite simple groups, then
$G/G_1$ is locally finite and the result follows from  \cite[Lemma 5.6]{KS}. 
If $G/G_1$ is prosoluble, then so is $G$, and then by   \cite[Proposition 5.12]{KS} $G$ has a 
series of finite length with pronilpotent quotients. In
this case, $w(G)$ is locally finite by  \cite[Lemma 5.7]{KS}, as required.
\end{proof}

\begin{proposition}\label{locfin}
Let $w$ be a multilinear commutator word and let $G$ be a profinite group  with restricted centralizers of $w$-values. 
Assume that every $w$-value has finite order.
 Then $w(G)$ is locally finite.
\end{proposition} 
\begin{proof} 
By Lemma \ref{lem:delta_k} there exists an integer $k$  such that each $\delta_k$-value is a $w$-value.  Set $u=\delta_{2k}$. 
 Let us show that  $G\in \Y_u$, that is,  
$$U(P)=\langle P_{u}^G\cap P\rangle$$ 
 is periodic for every Sylow subgroup of $P$ of $G$.

Let $P$ be a Sylow subgroup of $G$. It follows from Theorem \ref{pro-p} that $w(P)'$ is a finite $p$-group, so 
 $w(P)$ is soluble. 
  In view of Lemma \ref{lem:delta_k}  we have 
  $P^{(k)}\le w(P)$ and so  $P$  is soluble.
 By Proposition \ref{2015}, $P^{(k)}$ is locally finite and has finite exponent. In
particular $P^{(k)}$  is locally nilpotent.

If $P$ is finite then also $U(P)$ is finite so we can assume that $P$ is infinite. 

If some element  $x\in P_{\delta_k}$ has finite centralizer we get a contradiction,
because on the one hand $x^P$ is infinite, on the other hand $x^P$ is contained in $P^{(k)}$, which is finite by Lemma \ref{114}.
Thus we can assume that the centralizer of  each element in $P_{\delta_k}$ is infinite.  As $G$ has  restricted centralizers of $w$-values and every
$\delta_k$-value is also a $w$-value, it follows that each element in  $P_{\delta_k}$  has centralizer of finite index in $G$. 
Consider the sets
$$C_j=\{(\z_1,\dots,\z_{2^k}) \mid  \z_i\in P{\textrm{ and }} |\delta_k(\z_1,\dots,\z_{2^k})^G|\le j\}.$$
Note that each set  $C_j$ is closed. Moreover their union is the whole Cartesian product of $2^k$ copies of $P$.
By the Baire category theorem at least one of the sets $C_j$ has nonempty interior. 
Hence, there exist a natural number $m$, some elements
$a_i\in P$ and an  open normal subgroup
$T$ of $P$ such that 
$$\X_{\delta_k}(a_1T,\dots,a_{2^k}T)\subseteq C_{m}.$$
We deduce from Lemma \ref{comb1} that there exists a positive integer $m_1$ such that each element in $T_{\delta_k}$ 
has at most $m_1$ conjugates.
Let $T_0=T\cap P^{(k)}$. As $P^{(k)}$ is topologically 
generated by $P_{\delta_k}$, we can choose a right transversal   $b_1,\dots, b_r$ of $T_0$ in $P^{(k)}$ 
consisting of finite products of elements in $P_{\delta_k}$. Of course
 $b_1,\dots, b_r$ are FC-elements and thus there exists a positive integer $m_2$ such that each 
 $b_i$ has at most $m_2$ conjugates.
Let $x\in P_{u}$. 
We have $$x=\delta_k(c_1,\dots,c_{2^k}),$$  
where $c_i\in P_{\delta_k}$ for $i=1,\dots,2^k$. Now each $c_i$ is of the form $c_i=g_ih_i$ where
$g_i\in\{b_1,\dots,b_r\}$ and $h_i\in T_0$.

It follows from Lemma \ref{comb2} that $x=ah$ where $a$ is the product of at most $t_{2^k}$ conjugates of elements in
$\{b_1^{\pm 1},\dots,b_r^{\pm 1}\}$ and $h\in T_{\delta_k}$.

As each $b_i$ has at most $m_2$ conjugates and $h$ has at most $m_1$ conjugates it follows that $x$ has at most $m_3$ conjugates for some positive integer $m_3$ 
which does not depend on $x$.
So each $x\in P_{u}$ has order dividing $e$, where $e$ is the exponent of 
 $P^{(k)}$, and has at most $m_3$ conjugates. 

Recall that $U(P)=\langle P_{u}^G\cap P\rangle$. It follows from Lemma \ref{sol1} that $U(P)$  has 
finite exponent.  This proves that $G\in\Y_u$.

We deduce from Theorem \ref{prop-KS} that $G^{(2k)}$ is locally finite.  
Thus we can pass to the quotient group $G/G^{(2k)}$ and assume that
$G^{(2k)}=1$.  Now the result follows from Proposition \ref{2015}.
\end{proof}

\section{Proof of Theorem \ref{main}}\label{sec:final}

We recall that the Hirsch-Plotkin radical of an (abstract) group is defined  as    the maximal normal locally nilpotent subgroup.
   In a profinite group the Hirsch-Plotkin radical  need not be closed. 
  However, in the particular case where 
   the profinite group  is locally finite, the Hirsch-Plotkin radical is closed. 
  Indeed  the closure  of an abstract  locally nilpotent subgroup 
   is pronilpotent in any profinite group, and  so  it is
    locally nilpotent if the group  is  locally finite.  
   
   An important result about profinite torsion groups is the following theorem due to J. S. Wilson. 
   \begin{theorem}\cite[Theorem 1]{Will:torsion}\label{thm:wil}
   Let $G$ be a compact Hausdorff torsion group. Then $G$ has
a finite series
   \[ 1=G_0 \le G_1 \le \dots \le G_s \le G_{s+1}=G \] 
of closed characteristic subgroups, in which each factor $G_{i+1}/G_{i}$ either is a pro-$p$ group for some prime $p$ or 
is isomorphic (as a topological
group) to a Cartesian product of 
 finite simple groups. 
   \end{theorem}
   
   In particular, a profinite locally soluble torsion group has a finite series of characteristic subgroups 
   in which each factor  is a pro-$p$ group for some prime $p$.

\begin{proof}[Proof of Theorem \ref{main}]
 Recall that $w$ is a multilinear commutator word and $G$ a profinite group with  restricted centralizers of $w$-values. We want to prove that 
  $w(G)$ is abelian-by-finite. 
 
 If $G$ has a $w$-value of infinite order, then by Corollary \ref{infinite} the subgroup $w(G)$ is abelian-by-finite.
  So we can assume that every $w$-value has finite order.  
 It follows from Proposition \ref{locfin} that $w(G)$ is locally finite. 
 
 By Theorem \ref{thm:wil}, $w(G)$  has a finite series of characteristic subgroups
   \[ 1=A_0 \le A_1 \le \dots \le A_s \le A_{s+1}=w(G) \] 
   in which each factor 
 either is a pro-$p$ group for some prime $p$ or is isomorphic 
  to a Cartesian product of 
 finite simple groups. Let $A/B$ be a factor in the series which is isomorphic to a Cartesian product of 
 finite simple groups. 
 Recall that the famous Ore's conjecture, stating that every element of a nonabelian finite simple group is a commutator, was proved 
 in \cite{lost}. It follows that 
  every element of a nonabelian  finite simple group is a $w$-value, 
 therefore every element in $A/B$ is a $w$-value.  We deduce from Lemma \ref{KK}
 that $A/B$ is a profinite  group with restricted centralizers. By Shalev's result \cite{shalev},  $A/B$ is abelian-by-finite and therefore finite.
 
 Since all non-pronilpotent factors in the above series are finite, we derive that $w(G)$ is prosoluble-by-finite. Moreover
 $w(G)$ has an open characteristic subgroup $K$, which in turn has  a finite characteristic series 
$$1= F_0 \le F_1 \le F_2 \le \dots \le F_r \le F_{r+1} = K $$  
 where $F_{i+1}/F_i$ is the  Hirsch-Plotkin radical of $K/F_i$, for every $i$. 
 
 Alternatively, the existence of such a subgroup $K$ could be shown using theorems of Hartley \cite{H} 
 and Dade \cite{Dade}.

 Let $j$ be the maximal index such that all $w$-values contained in $F_j$ are FC-elements.
  If $j=r+1$, then by Corollary \ref{profinite-FC} we conclude that $w(G)$ is finite-by-abelian, hence abelian-by-finite. 
  
  So assume now that $j \le r$. Then  there exists a $w$-value whose centralizer  in $G$ is finite.   
  As $w(G)$ is locally finite,  Lemma \ref{KK} 
  guarantees that  $F_{j+1}/F_j$ has an element with finite centralizer. 
 Thus $F_{j+1}/F_j$ satisfies the hypothesis of Lemma \ref{114}, hence it is finite. Since $F_{j+1}/F_j$ is the  Hirsch-Plotkin radical of $K/F_j$, 
 it contains its  centralizer in $K/F_j$. 
  Taking into account that $F_{j+1}/F_j$ is finite, we conclude that its centralizer in $K$ has finite index. 
  Therefore $F_{j+1} $ has finite index in $K$.  
 We deduce that $F_j$ has finite index in $w(G)$. 
 
 Let $T$ be the $w^*$-residual of $F_j$.  Since every $w$-value in $F_j$ is an FC-element, we can apply Theorem \ref{genN}
 and we obtain that $T'$ is finite.  Hence, $T$ is abelian-by-finite. 
Note that $F_j/T$ is contained in $w^*(G/T)$, hence it centralizes $w(G/T)$. 
 By Lemma \ref{KK} the verbal subgroup $w(G/T)$ has an element with finite centralizer, so we deduce that $F_j/T$ is finite. 
 Thus $T$ is open in $w(G)$ and we conclude that $w(G)$ is abelian-by-finite, as desired. 
\end{proof}

In the sequel, we will use the fact that  an abelian-by-finite group contains a characteristic abelian subgroup of finite index
(see  \cite[Ch. 12, Lemma 1.2]{Passman} or  \cite[Lemma 21.1.4]{Ka}).

\begin{proof}[Proof of Corollary \ref{openT}]
Recall that $w$ is a multilinear commutator word and $G$ a profinite group 
 in which  centralizers of $w$-values 
 are either finite or open. 
 If follows from Theorem \ref{main} that $w(G)$ is abelian-by-finite. In particular
$w(G)$ has an open characteristic abelian subgroup $N$. As $w(G)/N$ is finite, there exists an open normal subgroup $T$ of $G$ 
containing $N$, such that 
$T/N$ intersects $w(G)/N$ trivially. Since $w(T) \le T \cap w(G) \le N$, we conclude that $w(T)$ is abelian, as desired. The solubility of $T$ is immediate from 
Lemma \ref{lem:delta_k}.  
\end{proof}

\begin{proof}[Proof of Corollary \ref{profinite-finite}]
Recall that $w$ is a multilinear 
commutator word and $G$  a profinite group in which every $w$-value has finite centralizer.  Assume that $w(G) \neq 1$.
If follows from Theorem \ref{main} that $w(G)$ is abelian-by-finite. In particular, 
$w(G)$ has an open characteristic abelian subgroup $N$. If $N$ contains a nontrivial $w$-value, then $N$ is finite, by assumption. Therefore we can assume that $N \cap G_w=1$.  It follows from the remark following Lemma \ref{ts}   that $N$ is  contained in  $w^*(G)$. Since the marginal subgroup centralizes $w(G)$, we deduce that $N$ is finite. 
This proves that $w(G)$ is finite. Hence, $C_G(w(G))$ has finite index in $G$. We see that $C_G(w(G))$ is both finite and of finite index, which proves that $G$ is finite. 
\end{proof}

As a final remark, we point out that in \cite{shalev} Shalev actually proved that if $G$ is a profinite group with restricted  centralizers then $\Delta (G)$ has finite index in $G$ and finite commutator subgroup. Our proof of Theorem  \ref{main} implies that if  $w$ is a multilinear commutator word and $G$ a profinite group   with restricted centralizers of $w$-values, then 
 the closed subgroup generated by $G_w \cap \Delta(G)$ has finite index in $w(G)$ and finite commutator subgroup. 

\section*{Acknowledgements}
The third author was partially  supported by FAPDF and CNPq.

\label{lastpage}
\end{document}